\RequirePackage{fix-cm}
\documentclass[smallextended]{svjour3}       % onecolumn (second format)
\smartqed  % flush right qed marks, e.g. at end of proof
\usepackage{graphicx}
\usepackage{amsfonts}
\usepackage{amsmath}
\usepackage{amssymb}
\usepackage{enumerate}

\begin{document}

\title{Reproducing Kernel Hilbert Space vs. Frame Estimates%\thanks{Grants or other notes
}


\author{Palle E. T. Jorgensen      \and
        Myung-Sin Song %etc.
}


\institute{Palle E. T. Jorgensen \at
Department of Mathematics, The University of Iowa, Iowa City, IA52242, USA \\
              Tel.: +1-319-335-0782\\
              Fax: +1-319-335-0627\\
              \email{jorgen@math.uiowa.edu}           %  \\
           \and
           Myung-Sin Song \at
              Department of Mathematics and Statistics, Southern Illinois University Edwardsville, Edwardsville, IL62026, USA \\
    Tel.: +1-618-650-2580\\
              Fax: +1-618-650-3771\\
\email{msong@siue.edu}
}

\date{Received: Mar 2015 / Accepted: date}

\maketitle

\begin{abstract}

We consider conditions on a given system $\mathcal{F}$ of vectors in Hilbert space 
$\mathcal{H}$, forming a frame, which turn $\mathcal{H}$ into a reproducing kernel 
Hilbert space. It is assumed that the vectors in $\mathcal{F}$ are functions on some 
set $\Omega$. We then identify conditions on these functions which automatically give 
$\mathcal{H}$ the structure of a reproducing kernel Hilbert space of functions on 
$\Omega$. We further give an explicit formula for the kernel, and for the corresponding 
isometric isomorphism. Applications are given to Hilbert spaces associated to families 
of Gaussian processes.

%Insert your abstract here. Include keywords, PACS and mathematical
%subject classification numbers as needed.
\end{abstract}

%\begin{acknowledgements}
%\end{acknowledgements}

\begin{acknowledgements} The authors are please to acknowledge helpful discussions, both recent and 
not so recent, with John Benedetto, Ilwoo Cho, D. Dutkay, 
Keri Kornelson, Kathy Merrill, P. Muhly, Judy Packer, Erin Pearse, 
Steen Pedersen, Gabriel Picioroaga, Karen Shuman. 
\end{acknowledgements}

\section{Introduction}
\label{sec:1}

A reproducing kernel Hilbert space (RKHS) is a Hilbert space $\mathcal{H}$ of functions 
on a set, say $\Omega$, with the property that $f(t)$ is continuous in $f$ with respect 
to the norm in $\mathcal{H}$. There is then an associated kernel. It is called 
reproducing because it reproduces the function values for $f$ in $\mathcal{H}$. 
Reproducing kernels and their RKHSs arise as inverses of elliptic PDOs, as covariance 
kernels of stochastic processes, in the study of integral equations, in statistical 
learning theory, empirical risk minimization, as potential kernels, and 
as kernels reproducing classes of analytic functions, and in the study of fractals, 
to mention only some of the current applications. They were first introduced in the 
beginning of the 20ties century by Stanisław Zaremba and James Mercer, G\'{a}bor 
Szeg\"{o}, Stefan Bergman, and Salomon Bochner. The subject was given a global and 
systematic presentation by Nachman Aronszajn in the early 1950s. The literature is by 
now vast, and we refer to the following items from the literature, and the papers cited 
there \cite{Aro50}, \cite{AAMN13}, \cite{Yam13}, \cite{NACLC12}, \cite{Wat12}, 
      \cite{HJLLM13}. Our aim in the present 
paper is to point out an intriguing use of reproducing kernels in the study of frames 
in Hilbert space.

\section{An Explicit Isomorphism}
\label{sec:2}

Let $\mathcal{H}$ be a separable Hilbert space, and let 
$\{\varphi_{n}\}_{n \in \mathbb{N}}$ be a system of vectors in $\mathcal{H}$.  Then 
we shall study relations of $\mathcal{H}$ as a reproducing kernel Hilbert space (RKHS) 
subject to properties imposed on the system $\{\varphi_{n}\}_{n \in \mathbb{N}}$.  A
RKHS is a Hilbert space $\mathcal{H}$ of functions on some set $\Omega$ such that for 
all $t \in \Omega$, there is a (unique) $K_{t} \in \mathcal{H}$ with 
$f(t)=\langle K_{t}, f \rangle_{\mathcal{H}}$, for all $t \in \Omega$, for all 
$f \in \mathcal{H}$.  In the theorem below we study what systems of functions
\begin{equation}
\label{eq:2.1.1}
  \varphi_{n} \in \mathcal{H} \cap \{ \text{functions on some set } \Omega\} 
\end{equation}
yield RKHSs; i.e., if $\{\varphi_{n}\}_{n \in \mathbb{N}}$ satisfies (\ref{eq:2.1.1}), 
what additional conditions are required to guarantee that $\mathcal{H}$ is a RKHS?

Given $\{\varphi_{n}\}_{n \in \mathbb{N}} \subset \mathcal{H}$, we shall introduce the 
Gramian $G=(\langle \varphi_{i}, \varphi_{j}\rangle_{\mathcal{H}})$ considered as an 
$\infty \times \infty-$matrix. 

Under mild restrictions on $\{\varphi_{n}\}_{n \in \mathbb{N}}$, it turns out that 
$G$ defines an unbounded (generally) selfadjoint linear operator
\[
  l^{2} \overset{G}{\rightarrow} l^{2}
\]
\begin{equation}
\label{eq:2.1.2}
  (G(c_{j}))_{k}=\sum_{j \in \mathbb{N}} \langle \varphi_{k}, \varphi_{j} 
  \rangle_{\mathcal{H}}c_{j}.
\end{equation}
Let $\mathcal{F}$ denote finitely supported sequence with (\ref{eq:2.1.2}) defined on
all finitely supported sequence $(c_{j})$ $\mathcal{F}$
in $l^{2}$, i.e., $(c_{j}) \in \mathcal{F}$ if and only if there exists 
$n \in \mathbb{Z}_{+}$ such that $c_{j}=0$, for all $j \geq n$; but note that $n$ 
depends the sequences.  Denoting $\delta_{j}$ the canonical basis in $l^{2}$, 
$\delta_{j}(j)=\delta_{i,j}$, note $\mathcal{F}=span\{\delta_{j}|j \in \mathbb{N}\}$.

Further, note that the RHS in (\ref{eq:2.1.2}) is well defined when 
$\sum_{j}\left\vert\langle \varphi_{k}, \varphi_{j} \rangle_{\mathcal{H}}\right\vert^{2}
<\infty$, for all 
$k \in \mathbb{N}$.

\begin{theorem}
\label{T:2.1}
Suppose $\mathcal{H}$, $\{\varphi_{n}\}$ are given.  Assume that
\begin{enumerate} [(a)]
\item each $\varphi_{n}$ is a function on $\Omega$ where $\Omega$ is a given set
\item $\{\varphi_{n}\}$ is a frame in $\mathcal{H}$, see (\ref{eq:2.4}) and 
(\ref{eq:2.5}), and that 
\item $\{\varphi_{n}(t)\} \in l^{2}$, for all $t \in \Omega$
\end{enumerate}
then $\mathcal{H}$ is a reproducing kernel Hilbert space (RKHS) with kernel 
\begin{equation}
\label{eq:2.1.21}
  K^{G}(s,t)=\langle l(s), G^{-1}l(t) \rangle_{2} 
            =l(s)^{*}G^{-1}l(t), 
\end{equation}
where $l(t)=\{\varphi_{n}(t)\} \in l^{2}$, and where $G$ is the Gramian of 
$G=(\langle \varphi_{n}, \varphi_{m} \rangle_{\mathcal{H}})$.
Moreover, $G$ defines selfadjoint operator in $l^{2}$ with
dense domain, and we get an isometric isomorphism 
\begin{equation}
\label{eq:2.1.3}
  \mathcal{H}_{RK} \overset{T_{G}}{\rightarrow} \mathcal{H}
\end{equation}
$T_{G}(\sum_{n}c_{n}\varphi_{n})=T^{*}c$ where $T$ is the frame operator.
\end{theorem}
\begin{proof}
\textbf{Overview: }
Since $\{\varphi_{n}\} \subset \mathcal{H}$ is a frame, the Gramian 
$G_{mn}:=\langle \varphi_{m}, \varphi_{n}\rangle_{\mathcal{H}}$, is an 
$\infty \times \infty$ matrix defining a bounded operator $l^{2} \rightarrow l^{2}$, 
invertible with $(G^{-1})_{mn}$ such that 
\[
  \sum_{k=1}^{\infty}(G^{-1})_{mk}\langle \varphi_{k}, \varphi_{n} \rangle_{\mathcal{H}}
  =\delta_{m, n}
\]
and the reproducing kernel of 
$\mathcal{H}$ is $\sum_{m}\sum_{n}\overline{\varphi_{m}(s)}G_{mn}^{-1}\varphi_{n}(t)
=\langle l(s), G^{-1}l(t) \rangle_{2}$
\end{proof}

\begin{proof}
(details) By (\ref{eq:2.1.3}) and Lemma \ref{L:2.3}, 
the frame operators $T$ and $T^{*}$ are as follows: Given $\mathcal{H}$, 
$\{\varphi_{n}\}$, set 
\begin{equation}
\label{eq:2.1}
  \begin{cases}
    T:\mathcal{H} \rightarrow l^{2} \\  
    T^{*}:l^{2} \rightarrow \mathcal{H}
  \end{cases}
\end{equation}
to be the two linear operators
\[
  Tf=(\langle \varphi_{n}, f\rangle_{\mathcal{H}}) 
\]
and adjoint $T^{*}$ as follows:
\[
  T^{*}c=\sum_{n}c_{n}\varphi_{n}.
\]

\begin{lemma}
\label{L:2.1}
We have 
\begin{equation}
\label{eq:2.2} 
  \langle Tf, c \rangle_{l^{2}}=\langle f, T^{*}c \rangle_{\mathcal{H}}, \quad \text{and} \quad
  T^{*}Tf=\sum \langle \varphi_{n}, f \rangle \varphi_{n}, \quad \forall f \in 
  \mathcal{H}, \quad \forall c \in l^{2},
\end{equation}
\begin{equation}
\label{eq:2.3}  
  (TT^{*}c)_{n}=(Gc)_{n}=\sum_{m}G_{nm}c_{m}, \quad \forall c \in l^{2}.
\end{equation}
\end{lemma}
Do the real case first, then it is easy to extend to complex valued functions. 

Note that $TT^{*}$ is an operator in $l^{2}$, i.e.,
\[
  l^{2} \overset{TT^{*}}{\rightarrow} l^{2}.
\]
It has a matrix-representation as follows
\begin{equation}
\label{eq:2.3.1}
  (TT^{*})_{i,j}=\langle \delta_{i}, TT^{*}\delta_{j}\rangle_{l^{2}}
\end{equation}

\begin{lemma}
\label{L:2.1.1}
We have 
\begin{equation}
\label{eq:2.3.2}
  (TT^{*})_{i,j}=G_{i,j}=\langle \varphi_{i}, \varphi_{j}\rangle_{\mathcal{H}}, 
  \quad \forall (i,j)\in \mathbb{N} \times \mathbb{N}.
\end{equation}
\end{lemma}
\begin{proof}
By (\ref{eq:2.3.1}), we have
\begin{align*}
  (TT^{*})_{i,j}&=\langle \delta_{i}, TT^{*}\delta_{j} \rangle_{l^{2}} \\
  &=\langle T^{*}\delta_{i}, T^{*}\delta_{j} \rangle_{\mathcal{H}} \\
  &=\langle \varphi_{i}, \varphi_{j}\rangle_{\mathcal{H}} = G_{i,j}
\end{align*}
which is the desired conclusion (\ref{eq:2.3.2}).
\end{proof}

Both $T^{*}T$ and $TT^{*}$ are self-adjoint: If $B_{i}$, $i=1,2$ 
are the constants from the frame estimates, then:
\begin{equation}
\label{eq:2.4}
  B_{1}\|c\|_{2}^{2} \leq \|T^{*}c\|_{\mathcal{H}}^{2} \leq B_{2}\|c\|_{2}^{2} \quad 
  \forall c \in l^{2}, \quad \text{and}
\end{equation}
\begin{equation}
\label{eq:2.5}
  B_{1}\|f\|_{\mathcal{H}}^{2} \leq \|Tf\|_{l^{2}}^{2} \leq 
  B_{2}\|f\|_{\mathcal{H}}^{2} \quad 
  \forall f \in \mathcal{H}; 
\end{equation}
equivalently
\[
  B_{1}\|f\|_{\mathcal{H}}^{2} \leq 
  \sum_{n}| \langle \varphi_{n}, f \rangle_{\mathcal{H}}|^{2} 
  \leq B_{2}\|f\|_{\mathcal{H}}^{2}.
\]

Set
\begin{equation}
\label{eq:2.6}
  K(s,t)=\sum_{n=1}^{\infty}\varphi_{n}(s)^{*}\varphi_{n}(t)=l(s)^{*}l(t)
  =\langle l(s), l(t) \rangle_{2}
\end{equation}
We have
\[
  B_{1}I_{l^{2}} \leq TT^{*} \leq B_{2}I_{l^{2}}, \quad \text{and}
\]
\[
  B_{1}I_{\mathcal{H}} \leq T^{*}T \leq B_{2}I_{\mathcal{H}}.
\]

If $B_{1}=B_{2}=1$, then we say that $\{\varphi_{n}\}_{n\in \mathbb{N}}$ is a Parseval 
frame.

For the theory of frames and some of their applications, see e.g., 
\cite{CaKu04}, \cite{Chr03}, \cite{CaCh98} and the papers cited there.

By the polar-decomposition theorems, see e.g., \cite{La02} we conclude that there is a 
unitary isomorphism $u: \mathcal{H} \rightarrow l^{2}$ such that 
$T=u(T^{*}T)^{1/2}=(TT^{*})^{1/2}u$; and so in particular, the two s.a. operators 
$T^{*}T$ and $TT^{*}$ are unitarily equivalent.

\begin{definition}
\label{D:2.1}
\begin{equation}
\label{eq:2.7}
  l(t)=(\varphi_{n}(t))\in l^{2}.
\end{equation}  
\end{definition}

Therefore $(T^{*}T)^{-1/2}$ is well defined $\mathcal{H} \rightarrow \mathcal{H}$.
Now (\ref{eq:2.2}) holds if and only if 
\[
  f=\sum\langle(T^{*}T)^{-1/2}\varphi, f \rangle (T^{*}T)^{-1/2}\varphi_{n}
\]
or equivalently:
\begin{equation}
\label{eq:2.8}
  f=\sum\langle \psi_{n}, f \rangle_{\mathcal{H}}\psi_{n},
\end{equation}  
where
\begin{equation}
\label{eq:2.9}
  \psi_{n}:=(T^{*}T)^{-1/2}\varphi_{n}.
\end{equation} 
Here we used that $T^{*}T$ is a selfadjoint operator in $\mathcal{H}$, and it has a 
positive spectral lower bound; where $\{\varphi_{j}\}_{j \in \mathbb{N}}$ is assumed to 
be a frame.
\end{proof}

\begin{lemma}
\label{L:2.1.2}
There is an operator $L:\mathcal{H} \rightarrow \mathcal{H}$ (the Lax-Milgram operator) 
such that 
\begin{equation}
\label{eq:2.9.1}
  \sum_{n=1}^{\infty}\langle f, \varphi_{n} \rangle_{\mathcal{H}}
  \langle \varphi_{n} Lg \rangle_{\mathcal{H}}=\langle f,g \rangle_{\mathcal{H}}
\end{equation}
holds for all $f \in \mathcal{H}$.
\end{lemma}
\begin{proof}
We shall apply the Lax-Milgram lemma \cite{La02}, p. 57 to the sesquilinear form
\begin{equation}
\label{eq:2.9.2}
  \mathcal{B}(f,g)=\sum_{n=1}^{\infty}\langle f, \varphi_{n} \rangle_{\mathcal{H}}
  \langle \varphi_{n}, g \rangle_{\mathcal{H}}, \quad \forall f,g \mathcal{H}.
\end{equation}
Since $\{\varphi_{n}\}_{n=1}^{\infty}$ is given to be a frame in $\mathcal{H}$, then 
our frame-bounds $B_{1}>0$ and $B_{2}<\infty$ such that (\ref{eq:2.5}) holds.  
Introducing $\mathcal{B}$ from (\ref{eq:2.9.2}) this into 
\begin{equation}
\label{eq:2.9.3}
  B_{1}\|f\|_{\mathcal{H}}^{2} \leq \mathcal{B}(f,f) \leq B_{2}\|f\|_{\mathcal{H}}^{2}, 
  \quad \forall f \mathcal{H}.
\end{equation}
The existence of the operator $L$ as stated in (\ref{eq:2.9.1}) now follows from the 
Lax-Milgram lemma.
\end{proof}

\begin{corollary}
\label{C:2.1}
Let $\mathcal{H}$, $\{\varphi_{n}\}$, $T$, $T^{*}$ be as in Lemma \ref{L:2.2}; and let 
$L$ be the Lax-Milgram operator; then $L=(T^{*}T)^{-1}$.
\end{corollary}

\begin{lemma}
\label{L:2.1.3}
The kernel $K^{G}(\cdot, \cdot)$ on $\Omega \times \Omega$ from (\ref{eq:2.1.21}) is 
well-defined and positive definite.
\end{lemma}
\begin{proof}
We must show that all the finite double summations
\[
  \sum_{i}\sum_{j}\overline{c_{i}}c_{j}K^{G}(t_{i},t_{j})
\]
are $\geq 0$, whenever $(c_{i})$ is a finite system of coefficients, and $(t_{i})$ is 
a finite sample of points in $\Omega$.  Now fix $(c_{i})$ and $(t_{i})$ as specified, 
and, for $n \in \mathbb{N}$, set
\[
  F_{n}:=\sum_{i}c_{i}\varphi_{n}(t_{i}); 
\]
then we have the following:
\begin{align*}
  \sum_{i}\sum_{j}\overline{c_{i}}c_{j}K^{G}(t_{i},t_{j}) 
  &= \sum_{i}\sum_{j}\overline{c_{i}}c_{j}\langle l(t_{i}), G^{-1}l(t_{j})\rangle_{l^{2}} \\
  &= \sum_{i}\sum_{j}\overline{c_{i}}c_{j}\sum_{m}\sum_{n}\overline{\varphi_{m}(t_{i})}
  G^{-1}_{m,n}\varphi_{n}(t_{j}) \\
  &=\sum_{m}\sum_{n}\overline{F_{m}}G^{-1}_{m,n}F_{n} \geq 0. 
\end{align*}
\end{proof}

\begin{lemma}
\label{L:2.2}
We have the following:
\begin{equation}
\label{eq:2.10}
  \psi_{n}(t)=(G^{-1/2}\varphi)_{n}(t)=\sum_{m=1}^{\infty}(G_{nm}^{-1/2}\varphi_{m})(t)
  =G^{-1/2}l(t)_{n}, 
\end{equation} 
and these functions are in the RKHS of the kernel $K^{G}$ from (\ref{eq:2.1.21}).
\end{lemma}
\begin{proof}
Begin with (the frame identity):
\begin{equation}
\label{eq:2.11}
  (T^{*}T)\varphi_{n}\underset{by (\ref{eq:2.2})}{=} 
  \sum_{m=1}^{\infty} \langle \varphi_{m}, \varphi_{n} \rangle \varphi_{m}=(Gl)_{n}, 
  \quad \forall n \in \mathbb{Z}_{+}, \text{ where } l=
  \begin{bmatrix}
      \varphi_{1}  \\
      \varphi_{2}  \\
      \vdots
  \end{bmatrix}
\end{equation} 
if and only if 
\[
  (T^{*}T)l(t)=G(l(t)).
\]
Now approximate $\sqrt{x}$ with polynomials (Weierstrass), and we get
\begin{equation}
\label{eq:2.12}
  (T^{*}T)^{-1/2}l(t)=G^{-1/2}l(t).
\end{equation} 
Recall, $\psi_{n}=(T^{*}T)^{-1/2}\varphi_{n}$. $\psi_{n}(t)=(G^{-1/2}l(t))_{n}$.
Now rewrite (\ref{eq:2.8}) as
\begin{equation}
\label{eq:2.12.1}
  f(t)=\sum_{n=1}^{\infty}\langle \psi_{n},f \rangle_{\mathcal{H}}\psi_{n}(t) 
      =\sum_{n=1}^{\infty}\langle G^{-1/2}\varphi_{n}, f\rangle_{\mathcal{H}}
        (G^{-1/2}\varphi)_{n}(t) 
      =\langle K_{t}^{G},f\rangle
\end{equation}
where
\begin{align*}
  K_{t}^{G}&=\sum_{n=1}^{\infty}G^{-1/2}\varphi_{n}(\cdot)(G^{-1/2}\varphi_{n})(t) \\
  &\underset{by (\ref{eq:2.12})}{=} K^{G}(s,t)
  =\sum_{n=1}^{\infty}(G^{-1/2}\varphi_{n})(s)(G^{-1/2}\varphi_{n})(t) \\
  &=\langle G^{-1/2}l(s), G^{-1/2}l(t)\rangle_{2} \\
  &=\langle l(s), G^{-1}l(t)\rangle_{2}
\end{align*}
For the complex case, the result still holds, mutatis mutandis; one only needs to add 
the complex conjugations.
\end{proof}

Note that (\ref{eq:2.12.1}) is the reproducing property.  

\begin{corollary}
\label{C:2.1.1}
The function $(\psi_{n}(t))$ from (\ref{eq:2.10}) in Lemma \ref{L:2.2} satisfy
\begin{equation}
\label{eq:2.12.2}
  \sum_{n\in\mathbb{N}}\overline{\psi_{n}(s)}\psi_{n}(t)=K^{G}(s,t), \quad 
  \forall (s,t) \in \Omega \times \Omega.
\end{equation}
\end{corollary}
\begin{proof}
\begin{align*}
  LHS_{(\ref{eq:2.12.2})}&=\langle G^{-1/2}l(s), G^{-1/2}l(t)\rangle_{2} \\
  &= \langle l(s), (G^{-1/2})^{2}l(t)\rangle_{2} \\
  &=\langle l(s), G^{-1}(l(t)),\rangle_{2} \\
  &= K^{G}(s,t).
\end{align*}
\end{proof}

\begin{lemma}
\label{L:2.3}
The following isometric property holds:
\begin{align*}
  \left\Vert \sum_{n=1}c_{n}\varphi_{n}(\cdot) \right\Vert_{\mathcal{H}}^{2}
  &= \sum\sum c_{m}c_{n}\langle \varphi_{n}, \varphi_{m} \rangle_{\mathcal{H}} \\
  &= c^{T}Gc=\langle c,Gc \rangle_{2} \\
  &=\langle c, TT^{*}c\rangle_{2}=\left\Vert T^{*}c\right\Vert_{\mathcal{H}}^{2}, 
  \quad T^{*}c \in \mathcal{H}, \quad c \in l^{2} 
\end{align*}
where $T$ and $T^{*}$ are the frame operators 
$\mathcal{H} \overset{T}{\underset{T^{*}}{\longleftrightarrow}} 
l^{2}$, i.e., $Tf=(\langle \varphi_{n}, f\rangle_{\mathcal{H}})_{n}\in l^{2}$ 
\end{lemma}

\begin{corollary}
\label{C:2.2}
The Lax operator $L$ satisfies $Lf:=\sum_{n}({T^{*}}^{-1}f)_{n}\varphi_{n}(\cdot)$, 
for all $f \in \mathcal{H}$ and it is isometric $\mathcal{H} \rightarrow \mathcal{H}$.
\end{corollary}

\begin{example}
\label{Ex:2.1}
In the theorem, we assume that the given Hilbert space $\mathcal{H}$ has a frame 
$\{\varphi_{n}\}\subset \mathcal{H}$ consisting of functions on a set $\Omega$.  So 
this entails a lower, and an upper frame bound, i.e., $0<B_{1}\leq B_{2}<\infty$. 

The following example shows that the conclusion in the theorem is false if there is 
not a positive lower frame-bound.

Set $\mathcal{H}=L^{2}(0,1)$, $\Omega=(0,1)$ the open unit-inbound, and 
$\varphi_{n}(t)=t^{n}$, $n \in \{0\}\cup\mathbb{N}=\mathbb{N}_{0}$.  In this case, 
the Gramian 
\[  
  G_{nm}=\int_{0}^{1}x^{n+m}dx=\frac{1}{n+m+1}
\]
is the $\infty \times \infty$ Hilbert matrix, see (\cite{Ro58}, \cite{Ka57}, 
    \cite{Ta49}).  In this 
case it is known that there is an upper frame bound $B_{2}=\pi$, i.e.,
\[
  \sum_{n=0}^{\infty}\left\vert \int_{0}^{1}f(x)x^{n}dx\right\vert^{2} \leq 
  \pi \int_{0}^{1}|f(x)|^{2}dx;
\]
in fact, for the operator-norm, we have 
\[
  \|G\|_{l^{2}\rightarrow l^{2}}=\pi;
\]
but there is not a lower frame bound.
Moreover, $G$ define a selfadjoint operator in $l^{2}(\mathbb{N}_{0})$ with spectrum 
$[0, \pi]=$ the closed interval.  This implies that there cannot be a positive lower 
frame-bound.

Moreover, it is immediate by inspection that $\mathcal{H}=L^{2}(0,1)$ is not a 
RKHS.
\end{example}

\section{Frames and Gaussian Processes}
\label{sec:3}

In \cite{AlJo12} and \cite{AlJoLe11}, it was shown that for every positive Borel 
measure $\sigma$ on $\mathbb{R}$ such that 
\begin{equation}
\label{eq:3.1}
  \int_{\mathbb{R}}\frac{d\sigma(u)}{1+u^{2}}< \infty,
\end{equation}
there is a unique (up to measure isomorphism) Gaussian proess $X$ as follows:
\begin{enumerate} [(i)]
  \item $X=X_{\varphi}$ is indexed by the Schwartz-space 
  $\mathcal{S}=\mathcal{S}(\mathbb{R})$ of function $\varphi$ on $\mathbb{R}$, 
  i.e., $\varphi \in \mathcal{S} \iff$
  $\varphi \in \mathbb{C}^{\infty}$, and for all $N, M \in \mathbb{N}$ we have
  \begin{equation}
    \label{eq:3.2}
    \underset{m \leq M}{max} \underset{x \in \mathbb{R}}{sup}\left\vert x^{N}
    \left(\frac{\partial}{\partial x}\right)^{m}\varphi(x)\right\vert < \infty
  \end{equation}
  with $\mathbb{E}(X_{\varphi})=0$, and 
  $\mathbb{E}(X_{\varphi}^{2})=\int_{\mathbb{R}}|\hat{\varphi}|^{2}d\sigma$ for all 
  $\varphi \in \mathcal{S}$.
  \item Let $\Omega:=\mathcal{S}'=$ the dual $=$ the Schwartz space of all tempered 
  distribution, then $X_{\varphi}$ is defined on $\mathcal{S}'$, by 
  \[
    X_{\varphi}(w)=w(\varphi), \quad \varphi \in \mathcal{S}, \quad w \in \mathcal{S}'.
  \]
  It is real valued Gaussian random variable.
  \item \cite{AlJo12}, \cite{AlJoLe11} there is a unique measure $\mathbb{P}=\mathbb{P}_{\sigma}$ on $\mathcal{S}'$ 
  such that 
  \begin{enumerate} [(a)]
    \item $X_{\varphi}$ is Gaussian for all $\varphi \in \mathcal{S}$,
    \item 
    \begin{equation}
    \label{eq:3.3}  
      \mathbb{E}(X_{\varphi})=0, \quad \text{and}
    \end{equation}
    \begin{equation}
    \label{eq:3.4}  
      \mathbb{E}_{\sigma}(e^{iX_{\varphi}})=\int_{\mathcal{S}'}e^{iX_{\varphi}}
      d\mathbb{P}_{\sigma}=e^{-\frac{1}{2}\|\widehat{\varphi}\|_{\sigma}^{2}}
      =e^{-\frac{1}{2}\int_{\mathbb{R}}|\widehat{\varphi}(u)|^{2}d\sigma(u)}
    \end{equation}
    where $\widehat{\varphi}=$ the Fourier transform,
    \begin{equation}
    \label{eq:3.5}  
      \widehat{\varphi}(u)=\int_{\mathbb{R}}e^{ixu}\varphi(x)dx.
    \end{equation}
  \end{enumerate}
\end{enumerate}

\begin{theorem}
\label{T:3.1}
Let $\{f_{n}\}_{n\in \mathbb{N}}$ be a real-valued frame in $L^{2}(\mathbb{R}, \sigma)$
$(=\{f \text{ on } \mathbb{R}, \text{ such that } \|f\|_{\sigma}^{2}:=
\int_{\mathbb{R}}|f(u)|^{2}d\sigma(u)<\infty\})$ with frame bounds $a$, $b$ such that
$0<a\leq b<\infty$, so 
\begin{equation}
\label{eq:3.6}
  a\|f\|_{\sigma}^{2} \leq \sum_{n \in \mathbb{N}}\left\vert 
  \int_{\mathbb{R}}f(u)f_{n}(u)d\sigma(u)\right\vert^{2} \leq b\|f\|_{\sigma}^{2}.
\end{equation}
Let $\{B_{n}\}_{n \in \mathbb{N}}$ be a system of i.i.d. (independent identically 
distributed) $N(0,1)$ Gaussian random variables on 
$(\Omega, \mathcal{F}, \mathbb{P}_{\sigma})$, and set
\begin{equation}
\label{eq:3.7}
  Y_{\varphi}(\cdot)=\sum_{n\in\mathbb{N}}\langle f_{n}, \widehat{\varphi}
  \rangle_{L^{2}(\sigma)}B_{n}(\cdot), \quad \text{(Karhunen-Lo\`{e}ve);}
\end{equation}
then
\begin{equation}
\label{eq:3.8}
  a\text{ }\mathbb{E}_{\sigma}\left(|X_{\varphi}|^{2} \right) \leq 
  \mathbb{E}_{\sigma}\left(|Y_{\varphi}|^{2} \right) \leq 
  b\text{ }\mathbb{E}_{\sigma}\left(|X_{\varphi}|^{2} \right). 
\end{equation}
\end{theorem}
\begin{proof}
Using the i.i.d. $N(0,1)$ property of $\{B_{n}\}_{n\in \mathbb{N}}$, we get 
\begin{equation}
\label{eq:3.9}
  \mathbb{E}_{\sigma}\left(|Y_{\varphi}|^{2} \right)=\sum_{n\in\mathbb{N}}
  |\langle f_{n}, \widehat{\varphi} \rangle_{L^{2}(\sigma)}|^{2}.
\end{equation}
The desired conclusion (\ref{eq:3.8}) now follows from (\ref{eq:3.6}) combined with 
\begin{equation}
\label{eq:3.10}
  \mathbb{E}_{\sigma}\left(|X_{\varphi}|^{2} \right)
  =\|\widehat{\varphi}\|_{L^{2}(\sigma)}^{2}, \quad \text{see \cite{AlJo12}, 
    \cite{AlJoLe11}}
\end{equation}
while (\ref{eq:3.10}) is immediate from (\ref{eq:3.4}).
\end{proof}

\begin{corollary}
\label{C:3.1}
The property for $\{Y_{\varphi}\}_{\varphi \in \mathcal{S}}$ in (\ref{eq:3.7}) agrees 
with $\{X_{\varphi}\}_{\varphi \in \mathcal{S}}$ if and only if 
$\{f_{n}\}_{n \in \mathbb{N}}$ is a Parseval frame in $L^{2}(\sigma)$.
\end{corollary}
\begin{proof}
This follows from the Karhunen-Lo\`{e}ve theorem; see 
\cite{JoSo07}, \cite{AlJo12}, \cite{AlJoLe11}.
\end{proof}

\bibliographystyle{spmpsci}      % mathematics and physical sciences
%\bibliography{kernel}


\end{document}